%% file: main.tex
\documentclass[11pt]{article}

\usepackage[T1]{fontenc}
\usepackage[utf8]{inputenc}
\usepackage{lmodern}
\usepackage{amsmath,amssymb,amsthm,mathtools}
\usepackage{booktabs}
\usepackage{enumitem}
\usepackage[a4paper,margin=32mm]{geometry}
\usepackage{microtype}
\usepackage[hidelinks]{hyperref}
\hypersetup{
  pdftitle={Infinite rational distance sets in affine general position: constructions in every dimension},
  pdfauthor={Junkai Qiu},
  pdfsubject={Rational distance sets in affine and general position},
  pdfkeywords={rational distance sets, integral point sets, general position,
    Chebyshev polynomials, rational quadratic forms, cyclic polytopes}
}

\newtheorem{theorem}{Theorem}[section]
\newtheorem{proposition}[theorem]{Proposition}
\newtheorem{lemma}[theorem]{Lemma}
\newtheorem{corollary}[theorem]{Corollary}
\theoremstyle{definition}
\newtheorem{definition}[theorem]{Definition}

\theoremstyle{remark}
\newtheorem{remark}[theorem]{Remark}

\newcommand{\Q}{\mathbb Q}
\newcommand{\R}{\mathbb R}
\newcommand{\Z}{\mathbb Z}
\newcommand{\C}{\mathbb C}

\newcommand{\Nm}{\operatorname{N}}
\newcommand{\diag}{\operatorname{diag}}
\newcommand{\eps}{\varepsilon}
\newcommand{\ip}[2]{\left\langle #1,#2\right\rangle}
\newcommand{\norm}[1]{\left\lVert #1\right\rVert}

\title{Infinite rational distance sets in affine general position:\
constructions in every dimension}
\author{Junkai Qiu\\
\small School of Mathematical Sciences, Dalian University of Technology\\
\small Dalian 116024, China\\
\small \texttt{qjk@mail.dlut.edu.cn}}
\date{}

\begin{document}
\maketitle

\begin{abstract}
For every integer $d\geq1$, we construct a countably infinite set
$X_d\subset\mathbb R^d$ in affine general position, with all pairwise
distances rational.  When $d$ is odd, $X_d$ may also be chosen
so that no $d+2$ points lie on a common sphere.  The construction is
uniform in $d$: positive Chebyshev square decompositions produce harmonic
curves on spheres whose points corresponding to rational parameter values
have pairwise rational distances.  A divided-difference factorization of the affine
determinant shows that sufficiently short arcs are locally convex, and
stereographic projection produces the odd-dimensional examples.  We also construct infinite
rational distance sets in $\mathbb Q^d$ in affine general position for
every even $d$, and in general position for every $d\equiv1\pmod4$.
For every $d\geq1$ and $n\geq d+1$, taking and rescaling suitable finite
subsets gives $n$-point integral point sets in affine general position.
A suitable ordered choice yields integral-distance realizations of all
cyclic polytopes.  In dimension three, we give an explicit rational
parametrization and obtain
infinitely many pairwise non-similar primitive $n_3$-clusters for every
$n\geq4$.
\end{abstract}

\medskip
\noindent\textit{2020 Mathematics Subject Classification.}
Primary 52C10; Secondary 51K05, 11E12, 52B11.

\smallskip
\noindent\textit{Keywords.}
Rational distance sets, integral point sets, general position, Chebyshev
polynomials, rational quadratic forms, cyclic polytopes.

\input{01-introduction}
\input{02-chebyshev-and-convexity}
\input{03-all-dimensional-geometry}
\input{04-rational-euclidean-models}
\input{05-consequences-and-dimension-three}
\input{06-boundaries}

\bibliographystyle{alpha}
\bibliography{references}
\end{document}

%% file: 01-introduction.tex
\section{Introduction}

A subset of Euclidean space is a \emph{rational distance set} if every
distance determined by two of its points is rational.  The requirement is
weak enough to allow familiar infinite examples on lines and circles, but
strong enough to retain arithmetic rigidity.  The contrast with integral
distances is immediate: the theorem of Anning and Erd\H{o}s says that an
  infinite subset of a finite-dimensional Euclidean space with all pairwise
  distances integral is collinear \cite{AnningErdos1945}.  Thus the noncollinear integral
point sets considered below are necessarily finite.  They arise by scaling
finite subsets of our infinite rational distance sets, and the scaling
factor may depend on the chosen subset.  At the other extreme, the
Erd\H{o}s--Ulam
problem asks whether the plane contains a dense rational distance set and
remains open; see \cite{SolymosiDeZeeuw2010,Shaffaf2018,AscherBrauneTurchet2020,CorvajaTurchetZannier2025}.

Rather than metric density, our focus is a hereditary form of geometric
nondegeneracy: every $d+1$ distinct points are affinely independent.  This
condition is substantially stronger than requiring only that the whole set
span $\mathbb R^d$.  Steiger constructed arbitrarily large finite rational
distance sets with this weaker property \cite{Steiger1953}; see also the
infinite constructions of Anil Kumar \cite{AnilKumar2023}.  Avdeev obtained
full-dimensional integral point sets of arbitrary finite cardinality
\cite{Avdeev2020}.

\begin{definition}\label{def:positions}
Let $P\subset\R^d$.
\begin{enumerate}[label=\textup{(\roman*)}]
  \item $P$ is in \emph{affine general position}, also called
  \emph{semi-general position} in the integral-point-set literature
  \cite{KohnertKurz2009}, if
  every $d+1$ distinct points are affinely independent.
  \item $P$ is in \emph{general position} if it is in affine general
  position and no $d+2$ distinct points are cospherical.
\end{enumerate}
\end{definition}

Our first theorem gives the basic construction in every dimension.

\begin{theorem}\label{thm:geometric-main}
For every integer $d\geq1$ there is a countably infinite set
$X_d\subset\R^d$ such that
\[
  \norm{x-y}\in\Q\qquad(x,y\in X_d)
\]
and $X_d$ is in affine general position.  If $d$ is odd, $X_d$ may be
chosen in general position.
\end{theorem}

Thus in dimension three there is an infinite rational distance
set in which every four points are noncoplanar and no five points are
cospherical.  The cardinality is optimal in every dimension: every
full-dimensional rational distance set is at most countable.  Indeed,
choose affinely independent anchors $a_0,\ldots,a_d$ in such a set.  The
distance-coordinate map
$\Phi(x):=(\norm{x-a_i})_{i=0}^d\in\Q^{d+1}$ is injective.  For
$1\leq i\leq d$, the polarization identity
\[
 2\ip{x}{a_i-a_0}
 =\norm{x-a_0}^2-\norm{x-a_i}^2
  +\norm{a_i}^2-\norm{a_0}^2
\]
shows that $\Phi(x)$ determines all the inner products
$\ip{x}{a_i-a_0}$.  Since the vectors $a_i-a_0$ form a basis of $\R^d$,
these inner products determine $x$.  Thus the set injects into the
countable set $\Q^{d+1}$.

By choosing the rational parameters recursively, the set in
Theorem~\ref{thm:geometric-main} may also be made dense in its underlying
short real-analytic arc, with no two distinct unordered pairs determining
the same distance.  To see this, enumerate a countable basis of the
parameter interval and, at each stage, avoid the zero sets of finitely many
nonzero analytic functions encoding repetitions of previously obtained
distances.  For distinct previously chosen points $a$ and $b$, an identity
\[
 \norm{p(t)-a}^2=\norm{p(t)-b}^2
\]
would place the entire arc in their perpendicular-bisector hyperplane.
Likewise, an identity $\norm{p(t)-a}^2=L^2$ would place the arc on a sphere
centered at $a$.  In the even-dimensional source construction, subtracting
the equation of the source sphere would then place the arc in a hyperplane;
in the odd-dimensional construction, such an identity is excluded by
general position.

The coordinates furnished directly by Theorem~\ref{thm:geometric-main}
need not be rational.  Requiring a realization in $\Q^d$ equipped with the
standard Euclidean metric is a genuinely finer arithmetic problem.  We
solve two uniform families.

\begin{theorem}[Rational-coordinate refinements]\label{thm:arithmetic-main}
The following hold.
\begin{enumerate}[label=\textup{(\roman*)}]
  \item For every even $d\geq2$, there is an infinite set
  $X_d\subset\Q^d$ in affine general position, and all its pairwise
  distances are rational.
  \item For every $d\equiv1\pmod4$, there is an infinite set
  $X_d\subset\Q^d$ in general position, and all its pairwise distances are
  rational.
\end{enumerate}
\end{theorem}

The distinction between Theorems~\ref{thm:geometric-main} and
\ref{thm:arithmetic-main} is controlled by rational quadratic forms.  The
geometric construction naturally carries a positive-definite rational
quadratic form that is an orthogonal sum of two-dimensional diagonal
blocks.  To place the curve in $\Q^d$ with its standard Euclidean metric,
that form must have square determinant and trivial Hasse invariant at every
finite place.  The determinant square class, called the
\emph{characteristic} in the integral-point-set literature
\cite{Kurz2006}, records only the first condition and is not sufficient by
itself in higher dimensions.  We make this failure explicit by constructing
a full-dimensional nine-point integral point set in $\R^8$ of
characteristic $1$ which has no congruent realization in $\Q^8$; see
Proposition~\ref{prop:char-counterexample}.  A related application of
the same local classification is due to Kominers \cite{Kominers2026}, who
uses square classes, norm groups, and Hilbert symbols to study rational
realizations of regular simplices.  Here the quadratic forms are
instead constrained by Chebyshev chord identities along an entire harmonic
curve.

Finite subsets give consequences for integral point sets.  Here an
\emph{integral point set} in $\R^d$ is a finite full-dimensional set with
all pairwise distances integral; its coordinates need not lie in $\Z^d$.
By contrast, following Kohnert and Kurz \cite{KohnertKurz2009}, an
$n_d$-cluster is an $n$-point integral point set in general position that
admits a realization in $\Z^d$.  We call any finite integral point set
\emph{primitive} if the greatest common divisor of its nonzero pairwise
distances is $1$; for an $n_d$-cluster this agrees with the convention of
\cite{KurzNollRathbunSimmons2014}.

\begin{corollary}[Finite integral consequences]\label{cor:finite-main}
Let $d\geq1$ and $n\geq d+1$.  In each of the following cases, there is an
$n$-point integral point set $P$ with the stated ambient-space and position
properties:
\begin{enumerate}[label=\textup{(\roman*)}]
  \item $P\subset\R^d$ in affine general position, for every $d$;
  \item $P\subset\R^d$ in general position, if $d$ is odd;
  \item $P\subset\Z^d$ in affine general position, if $d$ is even;
  \item $P\subset\Z^d$ in general position, if $d\equiv1\pmod4$.
\end{enumerate}
\end{corollary}

\begin{proof}
Choose $n$ points from the appropriate set in
Theorems~\ref{thm:geometric-main} and
\ref{thm:arithmetic-main}.  For a finite rational distance set
$Y\subset\R^d$, choose a positive integer $L$ divisible by every distance
denominator.  Then $LY$ has integral pairwise distances.  If
$Y\subset\Q^d$, include the coordinate denominators as well; then
$LY\subset\Z^d$.  Scaling by a positive factor preserves affine
dependence, the property that no $d+2$ points lie on a common sphere, and
order type.
\end{proof}

\begin{remark}[Infinitely many primitive similarity types]
\label{rem:similarity-types}
Fix $d\geq2$ and $n\geq d+1$.  On one of the analytic arcs in the proof,
fix a set $A=\{a_1,\ldots,a_{n-1}\}$ of points and vary the last rational
parameter along the injective arc.  For a fixed unlabeled similarity type,
choose one representative.  The fixed distance $\norm{a_1-a_2}$ must,
after rescaling, correspond to one of the finitely many distances of the
representative, so only finitely many similarity ratios are possible.
Consequently, the distance vector from the moving point to any $d$
affinely independent points of $A$ has only finitely many possible values.
Prescribed distances to $d$ affinely independent centers in $\R^d$
determine at most two points, interchanged by reflection in their affine
hyperplane.  Hence each unlabeled similarity type has a finite parameter
fiber.  Since there are infinitely many rational parameters, infinitely
many similarity types occur.

For each resulting finite configuration, scale first by an integer that
clears all distance denominators and then by the reciprocal of the greatest
common divisor of the resulting integral distances.  This produces a
primitive integral point set without changing its similarity type or its
position properties.  We therefore obtain infinitely many pairwise
non-similar primitive $n$-point integral point sets in affine general
position, and in general position when $d$ is odd.  This is a non-lattice
analogue of the
Erd\H{o}s--Noll infinite-or-bust question for $n_d$-clusters
\cite{KohnertKurz2009,KurzNollRathbunSimmons2014}.
\end{remark}

Kohnert and Kurz observed that, for $d\geq3$, no integral point set in
semi-general position with at least $d+5$ points, and no $n_d$-cluster with
$n\geq d+5$, was then known \cite[p.~2107]{KohnertKurz2009}.
Corollary~\ref{cor:finite-main} answers the former existence question
uniformly, to the best of our knowledge.  The lattice general-position
question is answered here in the dimensions stated explicitly in
Corollary~\ref{cor:finite-main} and
Theorem~\ref{thm:primitive-intro}.

A second finite consequence controls not merely affine independence but the
full order type.  For a labeled configuration
$p_1,\ldots,p_n\in\R^d$, its order type is the sign pattern
\[
 \chi(i_0,\ldots,i_d)=
 \operatorname{sgn}\det
 \begin{pmatrix}
  1&p_{i_0}^{\mathsf T}\\
  \vdots&\vdots\\
  1&p_{i_d}^{\mathsf T}
 \end{pmatrix}
\]
over all ordered $(d+1)$-tuples of distinct indices.  We call this order type
\emph{alternating} if, after reversing the ambient orientation if
necessary, $\chi(i_0,\ldots,i_d)=+1$ whenever
$i_0<\cdots<i_d$.  Recall also that the cyclic polytope $C(n,d)$ is the
combinatorial type of the convex hull of $n$ points
$\mu_d(t_i)=(t_i,t_i^2,\ldots,t_i^d)$ on the moment curve, with
$t_1<\cdots<t_n$ \cite{Ziegler1995}.

\begin{theorem}[Integral-distance cyclic polytopes]\label{thm:cyclic-intro}
For all $d\geq2$ and $n\geq d+1$, there is a labeled configuration of
$n$ points in $\R^d$ with integral pairwise distances and the alternating
order type.  Consequently, its convex hull is a realization of the cyclic
polytope $C(n,d)$.
\end{theorem}

To the best of our knowledge, Theorem~\ref{thm:cyclic-intro} is new.

In dimension three the rational-coordinate construction can be written in
closed form, without an abstract rational isometry.  Together with the
three-dimensional rational-to-lattice embedding theorem of Marshall and
Perlis \cite[Theorem~6]{MarshallPerlis2013}, it gives the following
primitive result.

\begin{theorem}[Primitive clusters in dimension three]
\label{thm:primitive-intro}
For every $n\geq4$ there are infinitely many pairwise non-similar
primitive $n_3$-clusters.
\end{theorem}

\subsection*{Outline of the construction}

The proof uses four main ingredients.

First, a product identity for the Chebyshev polynomials of the second kind
allows us to choose a rational linear combination of even-indexed
polynomials with square equal to a positive rational combination of
squares.  Second, each
summand is represented by a two-dimensional harmonic block.  The resulting
harmonic curve has rational chord lengths whenever both circle parameters
are rational.  Third, the harmonic Wronskian
factors into
  two Vandermonde determinants.  A divided-difference argument uses the
  nonvanishing of this Wronskian to prove affine independence of any $d+1$ points on
a sufficiently short arc.  Finally, even-dimensional examples are obtained
directly on the source sphere, while stereographic projection from an
endpoint of the source arc produces the odd-dimensional examples.  The inverse image of a
sphere in the target hyperplane is an affine hyperplane section of the
source sphere.

Section~\ref{sec:chebyshev-convexity} develops the algebraic and local
convexity mechanisms.  Section~\ref{sec:geometry} proves
Theorem~\ref{thm:geometric-main}.  Section~\ref{sec:arithmetic} proves
Theorem~\ref{thm:arithmetic-main} by norm conics and rational quadratic
forms.  Section~\ref{sec:consequences} proves
Theorem~\ref{thm:cyclic-intro}, develops the explicit three-dimensional
curve, and proves Theorem~\ref{thm:primitive-intro}.  We finish with the
rational-realization criterion and the precise scope of the method.

%% file: 02-chebyshev-and-convexity.tex
\section{Chebyshev squares, rational chord lengths, and local convexity}
\label{sec:chebyshev-convexity}

For $j\geq0$, let $U_j$ denote the Chebyshev polynomial of the second kind,
normalized by
\[
  U_j(\cos\theta)=\frac{\sin((j+1)\theta)}{\sin\theta},
  \qquad \theta\notin\pi\Z,
\]
with the values at $x=\pm1$ understood by continuity, and set $U_{-1}=0$.
We use the classical product linearization formula obtained by specializing
the ultraspherical product formula at parameter $1$; see
\cite[Theorem~6.8.2]{AndrewsAskeyRoy1999}.  For $0\leq i\leq j$, the
specialization needed below telescopes to
\begin{equation}\label{eq:cheb-product}
 U_{2i}(x)U_{2j}(x)
 =U_{i+j}(x)^2-U_{j-i-1}(x)^2.
\end{equation}
Indeed, after putting $x=\cos\theta$ and multiplying by $\sin^2\theta$,
this is the elementary identity
\[
 \sin((2i+1)\theta)\sin((2j+1)\theta)
 =\sin^2((i+j+1)\theta)-\sin^2((j-i)\theta).
\]
Since both sides of \eqref{eq:cheb-product} are polynomials in $x$,
the identity holds for every $x$.

\subsection{Chebyshev square decompositions and rational chords}

Parametrize the rational points of the unit circle other than $-1$ by
\begin{equation}\label{eq:circle}
 z(t)=\frac{1+it}{1-it}
 =\frac{1-t^2}{1+t^2}+i\frac{2t}{1+t^2},\qquad t\in\Q.
\end{equation}
Classical positive linearization concerns expansions in the basis
$\{U_k\}$.  The decomposition below is of a different kind: it expresses
the square of a linear combination of even-indexed Chebyshev polynomials as
a strictly positive combination of the nonlinear family $\{U_k^2\}$.  The
following lemma is the algebraic--geometric input specific to our
construction.

\begin{lemma}
\label{lem:chebyshev-chord}
For every integer $r\geq1$, let
\[
 K_r=
 \begin{cases}
   \{0,1,\ldots,r-1\},&r\text{ odd},\\
   \{1,2,\ldots,r\},&r\text{ even}.
 \end{cases}
\]
Then there are $P_r\in\Q[x]$ and $\lambda_k\in\Q_{>0}$, $k\in K_r$,
with the following properties.
\begin{enumerate}[label=\textup{(\roman*)}]
\item
\begin{equation}\label{eq:master-square}
 P_r(x)^2=\sum_{k\in K_r}\lambda_kU_k(x)^2.
\end{equation}
After multiplying both sides of this identity by a positive rational square, one may
take $P_r\in\Z[x]$ and all $\lambda_k\in\Z_{>0}$.
\item For $|z|=1$, define
\begin{equation}\label{eq:Gamma}
 \Gamma_r(z)=
 \left(\sqrt{\lambda_k}\,\Re z^{2(k+1)},
       \sqrt{\lambda_k}\,\Im z^{2(k+1)}\right)_{k\in K_r}
 \in\R^{2r}.
\end{equation}
Its image lies on the sphere of squared radius
\begin{equation}\label{eq:rho}
 \rho_r=\sum_{k\in K_r}\lambda_k\in\Q_{>0}.
\end{equation}
If $z,w\in S^1\cap\Q(i)$ and $z\overline w=c+is$, then
\begin{equation}\label{eq:chord}
 \norm{\Gamma_r(z)-\Gamma_r(w)}=2|sP_r(c)|\in\Q.
\end{equation}
\end{enumerate}
\end{lemma}

\begin{proof}
For rational coefficients $b_0,\ldots,b_m$, put
\[
 B_m(x)=\sum_{j=0}^{m}b_jU_{2j}(x).
\]
Applying \eqref{eq:cheb-product} to the cross terms gives
\begin{equation}\label{eq:lambda-expansion}
 B_m(x)^2=\sum_{k=0}^{2m}\lambda_kU_k(x)^2,
\end{equation}
where
\begin{equation}\label{eq:lambda-formula}
 \lambda_k=
 \sum_{\substack{i+j=k\\0\leq i,j\leq m}}b_ib_j
 -2\sum_{\substack{j-i-1=k\\0\leq i<j\leq m}}b_ib_j.
\end{equation}
The first sum is over ordered pairs.  Since $U_k^2$ has degree $2k$,
the expansion is unique.  In particular,
\begin{equation}\label{eq:extreme-lambda}
 \lambda_0=b_0^2-2\sum_{j=0}^{m-1}b_jb_{j+1},
 \qquad
 \lambda_{2m-1}=2b_{m-1}b_m,\quad
 \lambda_{2m}=b_m^2.
\end{equation}
Here the last two formulas are used for $m\geq1$.

Suppose first that $r=2m+1$.  We construct positive rational
$b_0,\ldots,b_m$ for which all the weights
$\lambda_0,\ldots,\lambda_{2m}$ are positive.  For $m=0$, take
$B_0=U_0$.  Given the construction at level $m-1$, retain its positive
coefficients and append $b_m=\eps$, where $\eps>0$ is rational.  At
$\eps=0$, the weights through index $2m-2$ are precisely the old positive
weights.  Since every weight is a polynomial in the coefficients, they
remain positive for all sufficiently small positive rational $\eps$.
The two new weights are $2b_{m-1}\eps$ and $\eps^2$ by
\eqref{eq:extreme-lambda}.

Now let $r=2m$.  We construct positive rational $b_0,\ldots,b_m$ for
which $\lambda_0=0$ and $\lambda_1,\ldots,\lambda_{2m}>0$.  The base case
is
\begin{equation}\label{eq:base-cheb}
 (2U_0+U_2)^2=4U_1^2+U_2^2.
\end{equation}
For the induction step, where $m\geq2$, assume that positive
$b_0,\ldots,b_{m-1}$ have the required property at level $m-1$.  Set
$\widetilde b_m=\eps>0$, retain
$\widetilde b_j=b_j$ for $j\neq1,m$, and define
\[
 \widetilde b_1=
 \frac{b_0^2-2\sum_{j=2}^{m-1}\widetilde b_j\widetilde b_{j+1}}
      {2(b_0+\widetilde b_2)},
\]
with an empty sum interpreted as zero.  Substitution in the first formula
of \eqref{eq:extreme-lambda} gives
  $\widetilde\lambda_0=0$.  Moreover, extend the coefficient vector at
level $m-1$ by setting $b_m=0$.  The corresponding equality at level
$m-1$ gives
\[
 b_0^2=2b_1(b_0+b_2)+2\sum_{j=2}^{m-2}b_jb_{j+1},
\]
so $\widetilde b_1\to b_1>0$ as $\eps\to0$.  Hence all the old positive
weights, and all the coefficients $\widetilde b_j$, remain positive for a
sufficiently small positive rational $\eps$.  The two new top weights are
$2\widetilde b_{m-1}\eps$ and $\eps^2$.  After relabeling
$\widetilde b_j$ as $b_j$, take $P_r=B_m$; this proves (i) in both parity
cases.  Finally, choose a positive integer $L$ such that
$LP_r\in\Z[x]$ and $L^2\lambda_k\in\Z_{>0}$ for every $k$; replacing
$P_r$ by $LP_r$ and $\lambda_k$ by $L^2\lambda_k$ gives the integral form.

For (ii), the constant-norm assertion follows immediately from $|z|=1$.
Write $z\overline w=e^{i\phi}$, so $c=\cos\phi$ and $s=\sin\phi$.
For each block,
\[
 |z^{2(k+1)}-w^{2(k+1)}|^2=4\sin^2((k+1)\phi).
\]
Consequently,
\[
 \begin{aligned}
 \norm{\Gamma_r(z)-\Gamma_r(w)}^2
 &=4\sum_{k\in K_r}\lambda_k\sin^2((k+1)\phi)\\
 &=4s^2\sum_{k\in K_r}\lambda_kU_k(c)^2
  =4s^2P_r(c)^2.
 \end{aligned}
\]
Taking the nonnegative square root proves \eqref{eq:chord}.
\end{proof}

The identity \eqref{eq:base-cheb} is equivalent to
\[
 4\sin^2(2\theta)+\sin^2(3\theta)
 =\sin^2\theta\,(4\cos^2\theta+1)^2;
\]
it is the first nontrivial member of a uniform family.  Apart from the
choice of positive weights and frequencies, $\Gamma_r$ is a generalized
trigonometric moment curve;
compare Schoenberg's classical convex trigonometric curve
\cite{Schoenberg1954}.  For related local geometry of trigonometric moment
curves, compare \cite{BarvinokLeeNovik2013}.  The fact that chords joining
points corresponding to rational parameters have rational length, however,
is the arithmetic consequence of
Lemma~\ref{lem:chebyshev-chord}(ii), not a generic
moment-curve property.

\subsection{Short arcs and affine independence}

We use the following divided-difference lemma.  As the nodes coalesce, the
quotient tends to the corresponding Wronskian determinant.

\begin{lemma}\label{lem:short-arc}
Let $J\subset\R$ be an open interval containing $0$, and let
$\gamma:J\to\R^n$ be real analytic.  If
\[
 \det[\gamma'(0),\gamma''(0),\ldots,\gamma^{(n)}(0)]\neq0,
\]
then there exists $\delta>0$ with $[0,\delta]\subset J$ such that every
$n+1$ distinct points on $\gamma([0,\delta])$ are affinely independent.
\end{lemma}

\begin{proof}
Set $f_0(t)=1$ and $f_j(t)=\gamma_j(t)$ for $1\leq j\leq n$.  Newton
interpolation, or equivalently the determinant formula for divided
differences \cite[Eq.~(45)]{DeBoor2005}, gives the factorization
\[
 F(t_0,\ldots,t_n):=\det[f_j(t_i)]_{i,j=0}^{n}
 =\prod_{0\leq i<j\leq n}(t_j-t_i)G(t_0,\ldots,t_n).
\]
The quotient $G$ extends analytically across all diagonals, and the
continuity and repeated-node formulas
\cite[Eq.~(14) and Proposition~21]{DeBoor2005} give
\[
 G(t,\ldots,t)=
 \det\left[\frac{f_j^{(i)}(t)}{i!}\right]_{i,j=0}^{n}.
\]
Expanding the first column at $t=0$ yields
\[
 G(0,\ldots,0)=
 \frac{\det[\gamma'(0),\ldots,\gamma^{(n)}(0)]}
      {\prod_{j=1}^{n}j!}\neq0.
\]
After shrinking $\delta$ if necessary, $[0,\delta]\subset J$ and $G$ does
not vanish on $[0,\delta]^{n+1}$.  For distinct nodes the Vandermonde
factor is also nonzero, and hence $F(t_0,\ldots,t_n)\neq0$.  This
determinant condition is precisely the affine independence of the
corresponding points.
\end{proof}

Let
\[
 \gamma_r(\theta)=\Gamma_r(e^{i\theta}),
 \qquad \omega_k=2(k+1),\quad k\in K_r.
\]
At $\theta=0$, the odd derivatives have nonzero entries only in the sine
coordinates, while the even derivatives have nonzero entries only in the
cosine coordinates.  After reordering rows and columns, the Wronskian matrix
is block diagonal, and its two blocks are Vandermonde-type matrices in the
distinct numbers $\omega_k^2$.  A direct evaluation gives
\begin{equation}\label{eq:wronskian}
 \left|\det[\gamma_r'(0),\ldots,\gamma_r^{(2r)}(0)]\right|
 =\left(\prod_{k\in K_r}\lambda_k\omega_k^3\right)
  \prod_{\substack{i<j\\i,j\in K_r}}
  (\omega_j^2-\omega_i^2)^2>0.
\end{equation}
It follows from Lemma~\ref{lem:short-arc} that there exists
$\delta_r>0$ such that every $2r+1$ distinct points on
$\gamma_r([0,\delta_r])$ are affinely independent.  In particular, this
statement includes the endpoint
\begin{equation}\label{eq:pole}
 N=\gamma_r(0)=\Gamma_r(1).
\end{equation}

\begin{remark}
The determinant condition says that the chosen arc is locally convex in
the sense of projective differential geometry.  For our purposes, the
divided-difference proof has the advantage of controlling all choices of
distinct nodes in a single short interval; compare \cite{KarlinStudden1966}.
\end{remark}

%% file: 03-all-dimensional-geometry.tex
\section{The construction in \texorpdfstring{$\R^d$}{d-dimensional Euclidean space}}
\label{sec:geometry}

In this section we prove the real-coordinate statement of
Theorem~\ref{thm:geometric-main}.
The coordinates of the present curves need not be
rational; rational-coordinate refinements are deferred to
Section~\ref{sec:arithmetic}.

\subsection{\texorpdfstring{$d=2r$}{d=2r}}

Choose a positive rational number $\eta_r$ so small that
$2\arctan\eta_r<\delta_r$, where $\delta_r$ is the short-arc constant
chosen after \eqref{eq:wronskian}.  Define
\begin{equation}\label{eq:X-even}
 X_{2r}=\{\Gamma_r(z(t)):t\in\Q\cap(0,\eta_r)\}.
\end{equation}

\begin{proposition}\label{prop:even-geometry}
The set $X_{2r}\subset\R^{2r}$ is infinite, has rational
pairwise distances, and every $2r+1$ distinct points are affinely
independent.
\end{proposition}

\begin{proof}
The angular parameter of $z(t)$ is $2\arctan t$.  Thus all chosen points
lie on the locally convex arc of Section~\ref{sec:chebyshev-convexity}.
Their affine independence follows from Lemma~\ref{lem:short-arc} and
\eqref{eq:wronskian}, while their distances are rational by
Lemma~\ref{lem:chebyshev-chord}(ii).  After shrinking $\eta_r$ if
necessary, both the rational circle parametrization and the harmonic curve are
injective on the chosen interval, so the set is infinite.
\end{proof}

Every point of $X_{2r}$ lies on the sphere of squared radius $\rho_r$.
Thus Proposition~\ref{prop:even-geometry} proves the
affine-general-position assertion of
Theorem~\ref{thm:geometric-main} in even dimension.  Since all points of
$X_{2r}$ lie on the same sphere, this construction does not place
$X_{2r}$ in general position: every $2r+2$ of its points are
cospherical.

\subsection{\texorpdfstring{$d=2r-1$}{d=2r-1}: stereographic projection}

Let $S=\{q\in\R^{2r}:\norm{q}^2=\rho_r\}$, let $N$ be
the point \eqref{eq:pole}, and put $H=N^\perp$.  For
$q\in S\setminus\{N\}$ define
\begin{equation}\label{eq:stereo}
 \sigma_N(q)=N+
 \frac{\rho_r}{\rho_r-\ip{q}{N}}(q-N).
\end{equation}
The point $\sigma_N(q)$ is the intersection of the line through $N$ and
$q$ with the $(2r-1)$-dimensional Euclidean space $H$.

Take the points $q(t)=\Gamma_r(z(t))$ corresponding to rational parameters
on the associated open short arc and set
\begin{equation}\label{eq:X-odd}
 X_{2r-1}=\{\sigma_N(q(t)):t\in\Q\cap(0,\eta_r)\}
 \subset H\cong\R^{2r-1}.
\end{equation}

\begin{proposition}\label{prop:odd-geometry}
The set $X_{2r-1}\subset H\cong\R^{2r-1}$ is infinite, has rational
pairwise distances, every $2r$ distinct points are affinely independent,
and no $2r+1$ distinct points are cospherical.
\end{proposition}

\begin{proof}
For $q,w\in S\setminus\{N\}$, the identities
\[
 \rho_r-\ip{q}{N}=\frac12\norm{q-N}^2
\]
and
\[
 2\ip{q-N}{w-N}
 =\norm{q-N}^2+\norm{w-N}^2-\norm{q-w}^2
\]
give, after substitution in \eqref{eq:stereo},
\begin{equation}\label{eq:stereo-distance}
 \norm{\sigma_N(q)-\sigma_N(w)}
 =\frac{2\rho_r\norm{q-w}}
 {\norm{q-N}\norm{w-N}}.
\end{equation}
Every chord joining two source points $q(t)$ and $q(u)$, and every chord
joining a source point $q(t)$ to $N$, has rational length by
Lemma~\ref{lem:chebyshev-chord}(ii).  Since
$\rho_r\in\Q$, formula \eqref{eq:stereo-distance} gives rational
distances in $X_{2r-1}$.  After shrinking $\eta_r$ as in the
even-dimensional case, the source curve is injective.  Stereographic
projection is also injective on $S\setminus\{N\}$, because a line through
$N$ meets the sphere in at most one further point.  Hence $X_{2r-1}$ is
infinite.

Now take distinct source points $q_i=q(t_i)$ and write
$p_i=\sigma_N(q_i)$.  Equation \eqref{eq:stereo} gives
\begin{equation}\label{eq:radial}
 p_i-N=\alpha_i(q_i-N),\qquad \alpha_i\neq0.
\end{equation}
Because $p_i\in H$ whereas $N\notin H$, affine dependence of
$p_1,\ldots,p_{2r}$ is equivalent to linear dependence of
$p_1-N,\ldots,p_{2r}-N$.  By \eqref{eq:radial}, such a dependence
would make $q_1-N,\ldots,q_{2r}-N$ linearly dependent, contrary to the
affine independence of $N,q_1,\ldots,q_{2r}$ on the closed short source
arc.  The use of the endpoint $N$ here is not an additional genericity
assumption: Lemma~\ref{lem:short-arc} was applied to the closed arc
containing $N$.

It remains to show that no $2r+1$ projected points are cospherical.  Suppose that
$p=\sigma_N(q)$ lies on the sphere in $H$ with center $c\in H$ and
squared radius $s_0^2$, and put $a_0=\norm{c}^2-s_0^2$.  Its equation is
$\norm{p}^2-2\ip{p}{c}+a_0=0$.  Substitution of
\eqref{eq:stereo}, followed by multiplication by
$\rho_r-\ip{q}{N}$ and use of $\ip{c}{N}=0$, gives the affine hyperplane
equation
\begin{equation}\label{eq:inverse-plane}
 \ip{q}{(\rho_r-a_0)N-2\rho_rc}
 +\rho_r(\rho_r+a_0)=0.
\end{equation}
This equation defines a proper affine hyperplane.  Indeed, if its normal
vector vanished, the orthogonal decomposition
$\R^{2r}=\R N\mathbin{\perp}H$ would give
$c=0$ and $a_0=\rho_r$, whereas
$a_0=\norm{c}^2-s_0^2=-s_0^2\leq0$, contradicting $\rho_r>0$.
Thus, if $2r+1$ projected points were cospherical, their $2r+1$ source
points would lie in one proper affine hyperplane of $\R^{2r}$.  This
contradicts Lemma~\ref{lem:short-arc}.
\end{proof}

Propositions~\ref{prop:even-geometry} and
\ref{prop:odd-geometry} prove Theorem~\ref{thm:geometric-main}.

%% file: 04-rational-euclidean-models.tex
\section{Rational-coordinate constructions}
\label{sec:arithmetic}

We now prove the rational-coordinate refinements in
Theorem~\ref{thm:arithmetic-main}.  Unlike Section~\ref{sec:geometry},
all point sets constructed in this section lie in $\Q^d$, equipped with
the standard Euclidean quadratic form.  The amplitudes $\sqrt{\lambda_k}$ in
\eqref{eq:Gamma} initially place the geometric curve in a finite
extension of $\Q$.  We remove them either by replacing the unit circle
with a rational norm conic and identifying the resulting rational
quadratic form with the standard sum-of-squares form,
or by realizing suitable weights as Gaussian norms.

\subsection{The norm-conic chord construction}

Let $D\in\Q_{>0}$ and let $\eta^2=-D$.  The norm-one conic has the
rational parametrization
\begin{equation}\label{eq:norm-param}
 \zeta_D(u)=\frac{1+\eta u}{1-\eta u}
 =\frac{1-Du^2}{1+Du^2}+\eta\frac{2u}{1+Du^2},
 \qquad u\in\Q.
\end{equation}
Write
\begin{equation}\label{eq:C-T}
 \zeta_D(u)^{2(k+1)}=C_k(u)+\eta T_k(u),
 \qquad C_k,T_k\in\Q(u).
\end{equation}
Let $K\subset\Z_{\geq0}$ be finite.  For positive rational weights
$\lambda_k$, $k\in K$, define the diagonal quadratic form with paired
coefficients
\begin{equation}\label{eq:QD}
 Q_D=\mathop{\perp}_{k\in K}
 \left\langle\frac{\lambda_k}{D},\lambda_k\right\rangle
\end{equation}
on the rational coordinate space with blocks $(X_k,Y_k)$, and put
\[
 G_D(u)=(C_k(u),T_k(u))_{k\in K}.
\]
Here $\langle a_1,\ldots,a_s\rangle$ denotes the diagonal form
$\sum_{j=1}^s a_jx_j^2$, and $\perp$ denotes orthogonal sum.

\begin{lemma}
\label{lem:norm-chord}
Suppose
\[
 B(x)^2=\sum_{k\in K}\lambda_kU_k(x)^2.
\]
For $u,v\in\Q$, write
$\zeta_D(u)\zeta_D(v)^{-1}=c+\eta s$.  Then
\begin{equation}\label{eq:norm-chord}
 Q_D(G_D(u)-G_D(v))=4s^2B(c)^2.
\end{equation}
In particular, the distance induced by $Q_D$ is rational, since
\[
 \sqrt{Q_D(G_D(u)-G_D(v))}=2|sB(c)|\in\Q.
\]
There is a sufficiently short real interval on which any $2|K|+1$
distinct points of the curve are affinely independent.
\end{lemma}

\begin{proof}
Under the complex embedding $\Q(\eta)\hookrightarrow\C$ determined by
$\eta\mapsto i\sqrt D$, write
$\zeta_D(u)\zeta_D(v)^{-1}=e^{i\phi}$.  Then
$c=\cos\phi$ and $\sqrt D\,s=\sin\phi$.  The $k$th block
therefore contributes
\[
 4\frac{\lambda_k}{D}\sin^2((k+1)\phi)
 =4\lambda_ks^2U_k(c)^2
\]
to the squared chord.  Summing proves \eqref{eq:norm-chord}.

For the short-arc assertion, put
$\theta(u)=2\arctan(\sqrt D\,u)$.  After the invertible block changes
\[
 (C_k,T_k)\longmapsto
 \left(\sqrt{\frac{\lambda_k}{D}}\,C_k,
       \sqrt{\lambda_k}\,T_k\right),
\]
the real curve has the form
\[
 D^{-1/2}\gamma(\theta(u)),
\]
where $\gamma$ is the corresponding harmonic curve of
Section~\ref{sec:chebyshev-convexity}.  The same calculation as in
\eqref{eq:wronskian}, with $K_r$ replaced by the finite set $K$, makes
the Wronskian matrix block diagonal, with two Vandermonde-type blocks in the
distinct squared frequencies $4(k+1)^2$.  It is therefore nonsingular.  Since
$\theta'(0)=2\sqrt D\neq0$, the reparametrized curve also has nonzero
Wronskian at $0$, so Lemma~\ref{lem:short-arc} applies.
\end{proof}

\subsection{Rational realization of quadratic forms}

Put $r=|K|$ and $A=\prod_{k\in K}\lambda_k$.  For a place $v$ of $\Q$,
write $(a,b)_v$ for the Hilbert symbol and use the convention
\[
 \epsilon_v(\langle a_1,\ldots,a_n\rangle)
 :=\prod_{1\leq i<j\leq n}(a_i,a_j)_v
\]
for the Hasse invariant.  For a finite prime $p$, let $v_p$ be the
normalized valuation with $v_p(p)=1$.

\begin{lemma}\label{lem:paired-invariants}
For every place $v$,
\begin{align}
 \det Q_D&=A^2D^{-r},
 &\det Q_D&\equiv D^{-r}\pmod{\Q^{\times2}},
 \label{eq:QD-det}\\
 \epsilon_v(Q_D)&=(A,-1)_v(D,A)_v
 (D,-1)_v^{\binom r2}.
 \label{eq:QD-Hasse}
\end{align}
\end{lemma}

\begin{proof}
The determinant formula is immediate.  Inside the $k$th block,
\[
 \left(\frac{\lambda_k}{D},\lambda_k\right)_v
 = (\lambda_k,-1)_v(D,\lambda_k)_v.
\]
Multiplication over $k$ gives $(A,-1)_v(D,A)_v$.  Each block determinant
has square class $D$, so each pair of distinct blocks contributes
$(D,D)_v=(D,-1)_v$.  There are $\binom r2$ such pairs.
\end{proof}

\begin{lemma}\label{lem:HM}
Let $q$ be a positive-definite rational quadratic form of dimension $n$.
Then
\[
 q\cong_{\Q}I_n
\]
if and only if $\det q\in\Q^{\times2}$ and
$\epsilon_p(q)=1$ for every finite prime $p$.  When these conditions hold,
there is a terminating exact algorithm that computes a
matrix $T\in\operatorname{GL}_n(\Q)$ with $q(x)=\norm{Tx}_2^2$.
\end{lemma}

\begin{proof}
Dimension, determinant square class and Hasse invariant classify a
nondegenerate quadratic form over each completion $\Q_v$.  Positive
definiteness supplies the real condition.  Thus the stated hypotheses make
$q$ locally isometric to $I_n$ everywhere, and Hasse--Minkowski gives the
  global isometry
  \cite[Chapter~IV, Section~3.2, Theorem~8]{Serre1973}.  For a completely elementary
  effectivity argument, enumerate rational matrices $T$ and test the identity
  $T^{\mathsf T}T=G$, where $G$ is the Gram matrix of $q$; the global
  existence just proved guarantees termination.  In practice, exact rational
  diagonalization, isotropy and conic solvers are far more efficient; see
  \cite{CremonaRusin2003,Simon2005}.  This gives an effective procedure in
  every fixed dimension, but not a uniform closed-form expression for the
  isometry matrix.
\end{proof}

For $a,b\in\Q_{>0}$, we repeatedly use the norm interpretation
\begin{equation}\label{eq:norm-Hilbert}
 (a,-b)_v=1\text{ for every }v
 \quad\Longleftrightarrow\quad
 a\in\Nm_{\Q(\sqrt{-b})/\Q}(\Q(\sqrt{-b})^\times).
\end{equation}
This equivalence combines the local norm criterion
\cite[Chapter~III, Section~1.1, Proposition~1]{Serre1973} with the Hasse norm
theorem for quadratic extensions \cite[Theorem~26.8.23]{Voight2021}.

\subsection{Stereographic projection for rational quadratic forms}

The odd-dimensional rational-coordinate construction is first carried out
on the quadric associated with a rational quadratic form rather than on a
standard Euclidean sphere.  We record the projection statements with
respect to the metric induced by that quadratic form.  Let $V$ be an
$(n+1)$-dimensional rational vector space, let $Q$ be a positive-definite
rational quadratic form on $V$, and let $B$ be its associated symmetric
bilinear form, normalized by $Q(x)=B(x,x)$.  Put
\[
 V_{\R}=V\otimes_{\Q}\R
\]
and extend $Q$ and $B$ to $V_{\R}$.  Fix $N\in V$ with
$Q(N)=\rho\in\Q_{>0}$, and put
\[
 H=N^{\perp_Q}\subset V,\qquad h=Q|_H,\qquad
 H_{\R}=H\otimes_{\Q}\R.
\]

\begin{lemma}
\label{lem:Q-projection-transfer}
Let
\[
 S_Q(\rho):=\{x\in V_{\R}:Q(x)=\rho\}.
\]
For $x\in S_Q(\rho)$ with $x\neq N$, put
\begin{equation}\label{eq:Q-stereo}
 \sigma_N^Q(x)
 =N+\frac{\rho}{\rho-B(x,N)}(x-N)
 =\frac{\rho x-B(x,N)N}{\rho-B(x,N)}.
\end{equation}
Then $\sigma_N^Q$ maps $S_Q(\rho)\setminus\{N\}$ injectively into
$H_{\R}$, maps $S_Q(\rho)\cap V\setminus\{N\}$ into $H$, and, for
$x,y\in S_Q(\rho)\setminus\{N\}$,
\begin{equation}\label{eq:Q-stereo-distance}
 Q\bigl(\sigma_N^Q(x)-\sigma_N^Q(y)\bigr)
 =\frac{4\rho^2Q(x-y)}{Q(x-N)Q(y-N)}.
\end{equation}

Let $\delta>0$ and let
$\gamma:[0,\delta]\to S_Q(\rho)$ be an injective real-analytic arc with
$\gamma(0)=N$.  Suppose $\gamma(t)\in V$ for rational $t$, that
\[
 \sqrt{Q(\gamma(s)-\gamma(t))}\in\Q
 \qquad(s,t\in\Q\cap[0,\delta]),
\]
and that every $n+2$ distinct points of the closed arc are affinely
independent.  Then the infinite set
\[
 \{\sigma_N^Q(\gamma(t)):t\in\Q\cap(0,\delta)\}\subset H
\]
satisfies $\sqrt{h(p-q)}\in\Q$ for every pair of its points $p,q$, every
$n+1$ of its points are affinely independent, and no $n+2$ lie on a
common $h$-sphere in $H_{\R}$.  If
$(H,h)\cong_\Q(\Q^n,I_n)$, composition with a rational isometry gives the
same conclusions in standard $\Q^n$.
\end{lemma}

\begin{proof}
Since $Q(x)=Q(N)=\rho$,
\begin{equation}\label{eq:pole-chord-Q}
 \rho-B(x,N)=\frac12Q(x-N).
\end{equation}
Formula \eqref{eq:Q-stereo} and $B(N,N)=\rho$ show that
$\sigma_N^Q(x)\in H_{\R}$, and that $\sigma_N^Q(x)\in H$ whenever
$x\in S_Q(\rho)\cap V\setminus\{N\}$.  To see that
the projection is injective, let $p\in H_{\R}$.  The points of the line
through $N$ and $p$ have the form $N+t(p-N)$, and the condition for such
a point to lie in $S_Q(\rho)$ is
\[
 t\bigl((h(p)+\rho)t-2\rho\bigr)=0.
\]
Besides the pole $t=0$, there is exactly one root,
$t=2\rho/(h(p)+\rho)$.  Thus every point in the image has a unique
non-pole preimage.

Writing
$u=x-N$, $v=y-N$, $a=2\rho/Q(u)$ and $b=2\rho/Q(v)$, the difference of
the projected points is $au-bv$.  Expanding $Q(au-bv)$ and using
\[
 2B(u,v)=Q(u)+Q(v)-Q(u-v)
\]
gives \eqref{eq:Q-stereo-distance}.

For source points $q_i$ and images $p_i=\sigma_N^Q(q_i)$,
\[
 p_i-N=\alpha_i(q_i-N),\qquad \alpha_i\neq0.
\]
Because $p_i\in H_{\R}$ and $N\notin H_{\R}$, affine dependence of the $p_i$ is
equivalent to linear dependence of the $p_i-N$.  Individual radial
rescaling therefore shows that $p_1,\ldots,p_{n+1}$ are affinely
independent if and only if $N,q_1,\ldots,q_{n+1}$ are affinely
independent.

Next suppose that $p=\sigma_N^Q(x)$ lies on the $h$-sphere
$h(p-c)=R^2$, with $c\in H_{\R}$, and put $a_0=h(c)-R^2$.  Substitution of
\eqref{eq:Q-stereo} in
$Q(p)-2B(p,c)+a_0=0$, followed by multiplication by
$\rho-B(x,N)$, gives the affine equation
\begin{equation}\label{eq:Q-inverse-sphere}
 B\bigl(x,(\rho-a_0)N-2\rho c\bigr)
 +\rho(\rho+a_0)=0.
\end{equation}
It defines a proper hyperplane.  Indeed, if its normal vector vanished,
the orthogonal decomposition
$V_{\R}=\R N\mathbin{\perp_Q}H_{\R}$ would give $c=0$ and
$a_0=\rho>0$, whereas $a_0=Q(c)-R^2=-R^2\leq0$, a contradiction.

Taking square roots in \eqref{eq:Q-stereo-distance} expresses each
target distance as
\[
 \frac{2\rho\sqrt{Q(x-y)}}
 {\sqrt{Q(x-N)}\sqrt{Q(y-N)}},
\]
which is rational by hypothesis.  Affine independence follows from
the radial-rescaling equivalence.  The image is infinite because both
$\gamma$ and $\sigma_N^Q$ are injective.  If $n+2$ image points lay on a
common $h$-sphere,
\eqref{eq:Q-inverse-sphere} would put the corresponding source points
in a proper affine hyperplane, contradicting the closed-arc hypothesis.
Finally, a rational isometry from $(H,h)$ to $(\Q^n,I_n)$ preserves all
the asserted metric, affine, and spherical properties.
\end{proof}

\subsection{\texorpdfstring{$d=2r$}{d=2r}}

When the number of harmonic blocks is odd, we use simultaneous norm
conditions.  Let
\[
 \mathcal N_i=\Nm_{\Q(i)/\Q}(\Q(i)^\times)
 =\{a^2+b^2:a,b\in\Q,\ (a,b)\neq(0,0)\}.
\]

\begin{lemma}\label{lem:simultaneous-norms}
Let $\alpha_1,\ldots,\alpha_s\in\Q$ and let $J\subset\R$ be a nonempty
open interval such that $t>0$ and $1+\alpha_jt>0$ for every $t\in J$ and
every $j$.  The set of rational numbers $t\in J$ satisfying
\[
 t\in\mathcal N_i,
 \qquad 1+\alpha_jt\in\mathcal N_i\quad(1\leq j\leq s)
\]
is dense in $J$.
\end{lemma}

\begin{proof}
Discard the zero values of $\alpha_j$ and retain only one representative
of each remaining value.  The homogeneous linear forms
\[
 u,\qquad v,\qquad u+\alpha_jv
\]
are then nonzero and pairwise nonproportional.  Index these forms by $\nu$
and denote them by $L_\nu(u,v)$.  Theorem~1.3 of Browning and Matthiesen
\cite{BrowningMatthiesen2017}, specialized to two base variables and with
all the norm-form fields chosen to be $\Q(i)$, says that the smooth open
variety defined by
\begin{equation}\label{eq:BM-special}
 x_\nu^2+y_\nu^2=L_\nu(u,v)\neq0
 \qquad\text{for every }\nu
\end{equation}
for these linear forms satisfies the Hasse principle and weak
approximation.  Their hypotheses allow the same field to occur repeatedly;
the required nondegeneracy here is exactly pairwise nonproportionality of
the forms.

Fix an open subinterval $I\subset J$ and $t_0\in I$.  At the real place
take $u_\infty=1$ and $v_\infty=t_0$.  All right-hand sides are positive,
so they are sums of two real squares.  At a finite prime $p$, take
\[
 u_p=1,\qquad \widetilde v_p=p^{2N_p},
\]
where $N_p$ is large enough that
\[
 2N_p+v_p(\alpha_j)\geq
 \begin{cases}1,&p\neq2,\\3,&p=2\end{cases}
 \qquad\text{for every }j.
\]
The value $\widetilde v_p$ is a square, and
\[
 1+\alpha_jp^{2N_p}\in
 \begin{cases}
  1+p\Z_p\subset\Q_p^{\times2},&p\neq2,\\
  1+8\Z_2\subset\Q_2^{\times2},&p=2.
  \end{cases}
 \]
 by Hensel's lemma; cf. \cite[Chapter~II]{Serre1973}.
Thus the indicated open variety has a local point over every completion of
$\Q$.
The points just constructed may be taken integral at all finite places, so
they form an adelic point of the open variety.

Weak approximation now supplies a rational point with $(u,v)$-coordinates
sufficiently close to $(1,t_0)$ at the real place that $u>0$ and
$t=v/u\in I$.  Every one of $u$, $v$ and $u+\alpha_jv$ is a nonzero
global Gaussian norm.  Since $\mathcal N_i$ is a multiplicative subgroup
of $\Q^\times$,
\[
 t=\frac vu\in\mathcal N_i,
 \qquad
 1+\alpha_jt=\frac{u+\alpha_jv}{u}\in\mathcal N_i.
\]
Every open subinterval $I$ contains such a $t$, proving density.
\end{proof}

\begin{remark}[Effectivity of the norm choice]
No height bound for the rational point furnished by
Lemma~\ref{lem:simultaneous-norms} is claimed.  In each fixed
dimension, however, one may enumerate $t\in J\cap\Q$ and test Gaussian
norm membership by checking that every prime $p\equiv3\pmod4$ has even
valuation; this is the norm criterion for $\Q(i)/\Q$, equivalently the
rational two-square theorem \cite[pp.~13--23]{Grosswald1985}.  The lemma
guarantees that this exact search
terminates.
\end{remark}

\begin{proposition}[Rational-coordinate realizations in even dimensions]
\label{prop:even-rational}
For every $r\geq1$, there is an infinite set
$Y_{2r}\subset\Q^{2r}$ with rational pairwise distances such that every
$2r+1$ distinct points are affinely independent.
\end{proposition}

\begin{proof}
We treat the cases $r$ even and $r$ odd separately.

\smallskip
\noindent\emph{Case 1: $r$ is even.}
Write $r=2m$ and choose the positive decomposition in the even case
supplied by
Lemma~\ref{lem:chebyshev-chord}(i) for $r=2m$:
\begin{equation}\label{eq:even-block}
 B_m(x)^2=\sum_{k=1}^{2m}\lambda_kU_k(x)^2,
 \qquad \lambda_k\in\Q_{>0},
\end{equation}
and put $A=\prod_{k=1}^{2m}\lambda_k$.

For the case $r\equiv2\pmod4$, we need the following parity adjustment.

\smallskip
\noindent\emph{Claim 1.}
For every $m\geq2$ and $\sigma\in\{0,1\}$, the weights in
\eqref{eq:even-block} may be chosen so that
\[
 v_2(A)\equiv\sigma\pmod2.
\]
Indeed, the identity comes from
$B_m=\sum_{j=0}^{m}b_jU_{2j}$ subject to $\lambda_0=0$.  At $m=2$, take
\[
 b_0=2,\qquad b_2=\delta=\frac{2^N}{L},\qquad
 b_1=\frac{2}{2+\delta},
\]
where $N>1$ and $L$ is a sufficiently large odd integer.  The coefficient
formula \eqref{eq:lambda-formula} gives
\[
 \lambda_1=4(b_1-\delta),\quad
 \lambda_2=b_1^2+4\delta,\quad
 \lambda_3=2b_1\delta,\quad
 \lambda_4=\delta^2.
\]
Their valuations are $2,0,N+1,2N$, so $v_2(A)=3+3N$ and either parity is
obtained by choosing the parity of $N$.  Inductively append
$b_m=\delta=2^N/L$ and adjust
\[
 b_1(\delta)=
 \frac{b_0^2-2\sum_{j=2}^{m-2}b_jb_{j+1}-2b_{m-1}\delta}
 {2(b_0+b_2)}
\]
to preserve $\lambda_0=0$.  Each old weight is a rational function of
$\delta$, defined and nonzero at $\delta=0$.  Once $N=v_2(\delta)$
exceeds a finite threshold, every perturbation term has strictly larger
$2$-adic valuation than the corresponding old weight; the ultrametric
inequality therefore leaves all old valuations unchanged.  There remain
arbitrarily large choices of either parity for $N$.  The two new weights
$2b_{m-1}\delta$ and $\delta^2$ contribute
$1+v_2(b_{m-1})+3N$ to $v_2(A)$; the parity of this contribution can be
prescribed by the choice of $N$.  Increasing the odd $L$ then makes the
real perturbation small enough to preserve positivity.  This
proves Claim~1.

If $r\equiv0\pmod4$, take $D=A$.  Since $\binom r2$ is even,
Lemma~\ref{lem:paired-invariants} gives
\[
 \epsilon_v(Q_A)=(A,-1)_v(A,A)_v=1
\]
at every place, while its determinant is a square.  Hence
$Q_A\cong_\Q I_{2r}$ by Lemma~\ref{lem:HM}.

Suppose next that $r\equiv2\pmod4$.  If $r=2$, using
\eqref{eq:base-cheb} in the norm-conic construction with $D=1$ gives
$Q_1=\langle4,4,1,1\rangle\cong_\Q I_4$.  If $r>2$, use Claim~1 to make
$v_2(A)$ odd.  Write $A=a/b$ in lowest terms.  Since $v_2(ab)$ is odd,
the Gauss--Legendre three-square theorem
\cite[Chapter~IV, Appendix]{Serre1973} gives
\begin{equation}\label{eq:A-three}
 A=x^2+y^2+z^2,\qquad x,y,z\in\Q.
\end{equation}
Indeed, $ab$ is not of the excluded form $4^e(8\ell+7)$, and an integral
representation of $ab$ may be divided by $b^2$.  The odd valuation shows
that $A$ is not a rational square, so after permutation we may assume
$k=y^2+z^2>0$.  Put $D=kA$.  Because $\binom r2$ is odd,
\eqref{eq:QD-Hasse} reduces to
\[
 \epsilon_v(Q_D)=(A,-k)_v(k,-1)_v=1.
\]
Both Hilbert symbols are equal to $1$ because $A=x^2+k$ is a norm from
$\Q(\sqrt{-k})$ and $k=y^2+z^2$ is a Gaussian norm; see
\eqref{eq:norm-Hilbert}.  The determinant is again a square, and
$Q_D\cong_\Q I_{2r}$.

In either subcase choose $T_D\in\operatorname{GL}_{2r}(\Q)$ such that
$Q_D(x)=\norm{T_Dx}_2^2$.  By
Lemma~\ref{lem:norm-chord}, after restricting to a sufficiently short
real parameter interval $I$, the set
\[
 Y_{2r}:=\{T_DG_D(u):u\in\Q\cap I\}\subset\Q^{2r}
\]
has rational pairwise distances and every $2r+1$ distinct points are
affinely independent.  After shrinking $I$ further so that the
parametrized arc is injective, the set is infinite.

\smallskip
\noindent\emph{Case 2: $r$ is odd.}
Write $r=2m+1$.  We use the following consequence of
Lemma~\ref{lem:simultaneous-norms}.

\smallskip
\noindent\emph{Claim 2.}
There are positive rational $b_0,\ldots,b_m$ and
$\lambda_0,\ldots,\lambda_{2m}$, all in $\mathcal N_i$, such that
\begin{equation}\label{eq:Gaussian-Cheb}
 \left(\sum_{j=0}^{m}b_jU_{2j}(x)\right)^2
 =\sum_{k=0}^{2m}\lambda_kU_k(x)^2.
\end{equation}
For $m=0$ take $b_0=\lambda_0=1$.  Given the assertion at level $m-1$,
append $tU_{2m}$.  The product identity
\eqref{eq:cheb-product} gives
\begin{align*}
 \lambda_k(t)&=\lambda_k-2tb_{m-k-1},&&0\leq k\leq m-1,\\
 \lambda_k(t)&=\lambda_k+2tb_{k-m},&&m\leq k\leq2m-2,
\end{align*}
 together with the two new weights
$\lambda_{2m-1}(t)=2b_{m-1}t$ and $\lambda_{2m}(t)=t^2$.  On some interval
$J=(0,\delta)$ all old weights remain positive, and every ratio
$\lambda_k(t)/\lambda_k$ has the form $1+\alpha_kt$.
Lemma~\ref{lem:simultaneous-norms} supplies
$t\in J\cap\mathcal N_i$ for which all these ratios
are Gaussian norms.  Multiplicativity of $\mathcal N_i$ and the fact that
$2\in\mathcal N_i$ give the required norm conditions.  This proves
Claim~2.

Choose $a_k,c_k\in\Q$ with $\lambda_k=a_k^2+c_k^2$, and set
\[
 M_k=\begin{pmatrix}a_k&-c_k\\c_k&a_k\end{pmatrix},
 \qquad M_k^{\mathsf T}M_k=\lambda_kI_2,
 \qquad M=\diag(M_0,\ldots,M_{2m}).
\]
For $D=1$, the curve $G_1(u)$ is the harmonic curve arising from the
standard rational parametrization of the unit circle.  Hence
$MG_1(u)\in\Q^{2r}$ for $u\in\Q$, while its squared Euclidean chord
length is given by the $D=1$ case of
Lemma~\ref{lem:norm-chord}, together with
\eqref{eq:Gaussian-Cheb}.
The invertible block transformation preserves the nonzero Wronskian.
Lemma~\ref{lem:short-arc} therefore gives a sufficiently short
injective interval $I$ on which
\[
 Y_{2r}:=\{MG_1(u):u\in\Q\cap I\}
\]
is infinite, has rational pairwise distances, and every $2r+1$ distinct
points are affinely independent.
\end{proof}

\subsection{\texorpdfstring{$d\equiv1\pmod4$}{d congruent to 1 mod 4}}

Let $m\geq1$, $r=2m+1$ and $d=4m+1$.  The source curve now has the
weights $\lambda_0,\ldots,\lambda_{2m}$ in
\eqref{eq:lambda-expansion}.  We require an explicit normalized
choice.

\begin{lemma}
\label{lem:effective-weights}
Let $M$ be the least odd integer with
$M>\lfloor\log_2m\rfloor$, and put
\begin{equation}\label{eq:effective-data}
 L=2^{M+4}+1,\qquad \xi=\frac{2^M}{L},\qquad
 S=\sum_{j=1}^{m}\xi^{2j}.
\end{equation}
Define
\begin{equation}\label{eq:effective-b}
 b_0=\frac{1-S}{1+S},\qquad
 b_j=\frac{2\xi^j}{1+S}\quad(1\leq j\leq m).
\end{equation}
Then every weight $\lambda_0,\ldots,\lambda_{2m}$ is positive,
\begin{equation}\label{eq:rho-one}
 \sum_{k=0}^{2m}\lambda_k=\sum_{j=0}^{m}b_j^2=1,
\end{equation}
and, for $A=\prod_{k=0}^{2m}\lambda_k$,
\begin{equation}\label{eq:v2-A}
 v_2(A)=4m+Mm(2m+1)+2v_2(m!),
 \qquad v_2(A)\equiv m\pmod2.
\end{equation}
\end{lemma}

\begin{proof}
Put $c_0=1-S$, $c_j=2\xi^j$ and
$\lambda_k=\Lambda_k/(1+S)^2$.  Substitution in
\eqref{eq:lambda-formula} gives
\begin{align}
 \Lambda_0&=(1-S)^2-4\xi(1-S)
 -8\sum_{i=1}^{m-1}\xi^{2i+1},\label{eq:Lambda0}\\
 \Lambda_k&=4\xi^k\left(k-S-\xi(1-S)
 -2\sum_{i=1}^{m-k-1}\xi^{2i+1}\right),
 &&1\leq k<m,\label{eq:Lambdak}\\
 \Lambda_m&=4\xi^m(m-S),\label{eq:Lambdam}\\
 \Lambda_{m+j}&=4(m-j+1)\xi^{m+j},
 &&1\leq j\leq m.\label{eq:Lambdahigh}
\end{align}
Since $0<\xi<1/16$,
$S<\xi^2/(1-\xi^2)<1/255$.  Hence
\[
 \Lambda_0>
 \left(\frac{254}{255}\right)^2-\frac14-\frac1{510}>0,
\]
and the bracket in \eqref{eq:Lambdak} is at least
$1-1/255-1/16-1/2040>0$.  The remaining expressions are manifestly
positive.  Moreover $(1-S)^2+4S=(1+S)^2$, which proves
\eqref{eq:rho-one}.

The denominator $(1+S)^2$ is a $2$-adic unit.  Since
$M>v_2(k)$ for $1\leq k\leq m$, the lowest $2$-adic term in each bracket
above is unique, and
\begin{align*}
 v_2(\Lambda_0)&=0,\\
 v_2(\Lambda_k)&=2+Mk+v_2(k),&&1\leq k\leq m,\\
 v_2(\Lambda_{m+j})&=2+M(m+j)+v_2(m-j+1),
 &&1\leq j\leq m.
\end{align*}
Summing gives the first equality in \eqref{eq:v2-A}; since $M$ is
odd, reduction modulo $2$ gives the second.
\end{proof}

\begin{proposition}[Rational-coordinate realizations in dimensions
$d\equiv1\pmod4$]
\label{prop:odd-rational}
For every $m\geq0$, there is an infinite set
$Y_{4m+1}\subset\Q^{4m+1}$ with rational pairwise distances such that
every $4m+2$ distinct points are affinely independent and no $4m+3$
distinct points are cospherical.
\end{proposition}

\begin{proof}
For $m=0$, take $Y_1=\Q$.  Henceforth assume $m\geq1$ and take the
weights supplied by Lemma~\ref{lem:effective-weights}.

Choose the norm parameter
\begin{equation}\label{eq:D-odd}
 D=\begin{cases}
 A,&m\text{ even},\\
 kA,&m\text{ odd},
 \end{cases}
\end{equation}
where, when $m$ is odd, we choose rational numbers $x,y,z$ satisfying
\begin{equation}\label{eq:three-odd}
 A=x^2+y^2+z^2=x^2+k,
 \qquad k=y^2+z^2>0.
\end{equation}
Indeed, when $m$ is odd, \eqref{eq:v2-A} makes $v_2(A)$ odd.
  Writing $A=a/b$ in lowest terms, the Gauss--Legendre three-square theorem
  \cite[Chapter~IV, Appendix]{Serre1973} represents $ab$ as three integral
  squares.  Dividing by $b^2$ gives
\eqref{eq:three-odd}; a finite search produces the representation.

Use the norm-conic curve $G_D(u)$ from
\eqref{eq:norm-param}--\eqref{eq:C-T}, with indices
$0\leq k\leq2m$.  For each block,
\[
 C_k(u)^2+D T_k(u)^2=1.
\]
Hence \eqref{eq:rho-one} gives
\[
 Q_D(G_D(u))=\frac1D.
\]
Let $N=((1,0))_{k=0}^{2m}$, let $H=N^{\perp_{Q_D}}$, and put
$h_D=Q_D|_H$.  The orthogonal splitting is
\[
 Q_D\cong_\Q\left\langle\frac1D\right\rangle\perp h_D.
\]

\smallskip
We claim that, with $D$ chosen by \eqref{eq:D-odd},
$(H,h_D)\cong_\Q(\Q^{4m+1},I_{4m+1})$.
Indeed, the orthogonal splitting and
Lemma~\ref{lem:paired-invariants} give
\begin{equation}\label{eq:hdet}
 \det h_D\equiv\frac{A^2}{D^{2m}}\equiv1
 \pmod{\Q^{\times2}}.
\end{equation}
At every place $v$, the orthogonal-sum formula gives
\[
 \epsilon_v(Q_D)
 =\epsilon_v(h_D)(1/D,\det h_D)_v.
\]
The last Hilbert symbol is $1$ by \eqref{eq:hdet}; hence
\begin{equation}\label{eq:hHasse}
 \epsilon_v(h_D)=\epsilon_v(Q_D)
 =(A,-1)_v(D,A)_v(D,-1)_v^m.
\end{equation}
If $m$ is even, $D=A$ and
\[
 \epsilon_v(h_D)=(A,-1)_v(A,A)_v=1.
\]
If $m$ is odd, $D=kA$ and Hilbert-symbol identities reduce
\eqref{eq:hHasse} to
\[
 \epsilon_v(h_D)=(A,-k)_v(k,-1)_v=1.
\]
Both Hilbert symbols are equal to $1$ by the two norm representations in
\eqref{eq:three-odd} and the criterion
\eqref{eq:norm-Hilbert}.  By Lemma~\ref{lem:HM}, this proves the
claim.  Fix a rational isometry
\[
 \tau_m:(H,h_D)\longrightarrow(\Q^{4m+1},I_{4m+1}).
\]

Near $u=0$, an invertible block transformation and a locally invertible
analytic change of parameter identify the real source curve with the
corresponding harmonic curve.  Lemma~\ref{lem:short-arc}, applied to
the closed source arc, therefore gives $\delta>0$ such that
every $4m+3$ distinct points of
$\{G_D(u):0\leq u\leq\delta\}$ are affinely independent.  In particular,
the conclusion applies simultaneously to points of the open arc and to the
endpoint $N=G_D(0)$.  Lemma~\ref{lem:norm-chord} shows that every
chord joining two source points with rational parameters has rational
$Q_D$-length.  Apply Lemma~\ref{lem:Q-projection-transfer}
with
\[
 n=4m+1,\qquad Q=Q_D,\qquad \rho=\frac1D,
 \qquad \gamma(u)=G_D(u).
\]
It yields rational distances induced by $h_D$, shows that every $4m+2$
projected points are affinely independent, and shows that no $4m+3$
projected points lie on a common $h_D$-sphere.  Finally compose with
$\tau_m$.  The
resulting rational-parameter image is the required infinite set
$Y_{4m+1}\subset\Q^{4m+1}$.
\end{proof}

Propositions~\ref{prop:even-rational} and
\ref{prop:odd-rational} prove
Theorem~\ref{thm:arithmetic-main}.

\begin{remark}
The Browning--Matthiesen theorem is used only in
Lemma~\ref{lem:simultaneous-norms}, hence only for the
$d\equiv2\pmod4$ rational-coordinate refinement in dimensions $d\geq6$;
the case $d=2$ does not use it.  The construction in $\R^d$, the
rational-coordinate realizations for $d\equiv0\pmod4$, and the
general-position realizations in $\Q^d$ for $d\equiv1\pmod4$ are
independent of that input.  We make no rational-coordinate general-position claim for
$d\equiv3\pmod4$.
\end{remark}

%% file: 05-consequences-and-dimension-three.tex
\section{Finite consequences and an explicit three-dimensional curve}
\label{sec:consequences}

\subsection{Cyclic polytopes}

\begin{proof}[Proof of Theorem~\ref{thm:cyclic-intro}]
Let $p_d:I_d\to\R^d$ denote the real-analytic parametrization used in
the proof of Theorem~\ref{thm:geometric-main}, restricted to its short
interval.  Choose ordered rational parameters $t_1<\cdots<t_n$ in $I_d$
and write $p_i=p_d(t_i)$.  For every increasing $(d+1)$-tuple of indices
$i_0<\cdots<i_d$, define
\[
 \Delta(t_{i_0},\ldots,t_{i_d})=
 \det\begin{pmatrix}
 1&p_d(t_{i_0})^{\mathsf T}\\
  \vdots&\vdots\\
 1&p_d(t_{i_d})^{\mathsf T}
 \end{pmatrix}.
\]
The determinant is nonzero throughout the connected chamber
\[
 \{(s_0,\ldots,s_d)\in I_d^{d+1}:s_0<\cdots<s_d\}.
\]
Its sign is therefore constant.  The ordered configuration has the
alternating chirotope, hence the alternating oriented matroid; its convex
hull is combinatorially equivalent to the cyclic polytope $C(n,d)$
\cite{BjornerEtAl1999,Ziegler1995}.  All pairwise distances are rational,
and a single scaling factor makes them integral without changing the oriented
matroid or the combinatorial type.
\end{proof}

In the rational-coordinate cases of Theorem~\ref{thm:arithmetic-main}, the
argument in the proof of Theorem~\ref{thm:cyclic-intro}, followed by one
dilation, puts the realization in $\Z^d$.  Thus, for every even $d$ and
every $d\equiv1\pmod4$, $C(n,d)$ has a lattice realization in which all
pairwise distances are integers.

\subsection{An explicit rational curve in three dimensions}

In dimension three, the rational isometry can be made explicit.

\paragraph{Construction of the curve.}
The identity \eqref{eq:base-cheb} has frequencies $2,3$ and weights
$(4,1)$.  For a general norm parameter $D$, the source quadratic form and
the $Q_D$-orthogonal complement of the pole $N=(1,0,1,0)$ are
\[
 Q_D=\left\langle\frac4D,4\right\rangle
 \perp\left\langle\frac1D,1\right\rangle,
 \qquad
 N^{\perp_{Q_D}}=\{(a,b,-4a,e)\},
\]
with
\[
 Q_D(a,b,-4a,e)=\frac{20}{D}a^2+4b^2+e^2.
\]
We take $D=5$.  Then $Q_5(N)=1$, and
\[
 J(a,b,-4a,e)=(2a,2b,e)
\]
is a rational isometry from
$(N^{\perp_{Q_5}},Q_5|_{N^{\perp_{Q_5}}})$ to $(\Q^3,I_3)$, the standard
sum-of-three-squares form.

Let $\eta^2=-5$ and, for a real parameter $u$, put
\begin{equation}\label{eq:csF}
 c(u)=\frac{1-5u^2}{1+5u^2},\qquad
 s(u)=\frac{2u}{1+5u^2},\qquad F(u)=4c(u)^2+1.
\end{equation}
For rational $u$, all three functions in \eqref{eq:csF} are rational.
Moreover,
\[
 \zeta(u)=\frac{1+\eta u}{1-\eta u}=c(u)+\eta s(u),
 \qquad c(u)^2+5s(u)^2=1.
\]
Write
\[
 \zeta(u)^4=C_1+\eta T_1,
 \qquad
 \zeta(u)^6=C_2+\eta T_2.
\]
Using $c^2+5s^2=1$ gives
\[
 \begin{aligned}
 C_1&=8c^4-8c^2+1,
 &T_1&=4cs(2c^2-1),\\
 C_2&=32c^6-48c^4+18c^2-1,
 &T_2&=2cs(16c^4-16c^2+3).
 \end{aligned}
\]
The source point $G(u)=(C_1,T_1,C_2,T_2)$ satisfies
$Q_5(G(u))=1$ and $G(0)=N$.  If
\[
 \Delta=5-4C_1-C_2,
\]
then the bilinear form associated with $Q_5$ satisfies
$B_5(G(u),N)=(4C_1+C_2)/5$, and
\eqref{eq:Q-stereo} gives
\[
 \sigma_N^{Q_5}(G(u))=
 \left(
 \frac{C_1-C_2}{\Delta},
 \frac{5T_1}{\Delta},
 -\frac{4(C_1-C_2)}{\Delta},
 \frac{5T_2}{\Delta}
 \right).
\]
The factorizations
\[
 \Delta=10s^2(4c^2+1)^2,
 \qquad
 C_1-C_2=10s^2(16c^4-12c^2+1)
\]
show that the projected point
$P(u):=J\sigma_N^{Q_5}(G(u))$ has the following rational coordinates:
\begin{equation}\label{eq:P3}
 P(u)=\left(
 \frac{2(16c^4-12c^2+1)}{F^2},
 \frac{4c(2c^2-1)}{sF^2},
 \frac{c(16c^4-16c^2+3)}{sF^2}
 \right),
\end{equation}
where $c,s,F$ are evaluated at $u$.

\begin{theorem}[Explicit three-dimensional curve]\label{thm:P3}
The set
\[
 \{P(u):u\in\Q\cap(0,1/6)\}\subset\Q^3
\]
is infinite and has rational pairwise distances.  Moreover, every four
distinct points are affinely independent, and no five distinct points are
cospherical.
\end{theorem}

\begin{proof}
\medskip\noindent\emph{Distances and injectivity.}
For distinct $u,v$, abbreviate $c_u=c(u)$ and $s_u=s(u)$, and put
\[
 C=c_uc_v+5s_us_v,
 \qquad
 S=s_uc_v-c_us_v.
\]
Lemma~\ref{lem:norm-chord}, specialized to
\eqref{eq:base-cheb}, gives
\[
 Q_5(G(u)-G(v))=4S^2(4C^2+1)^2,
 \qquad
 Q_5(G(w)-N)=4s_w^2(4c_w^2+1)^2.
\]
The projection distance formula \eqref{eq:Q-stereo-distance} and the
isometry $J$ therefore give the exact identity
\begin{equation}\label{eq:P3-distance}
 \norm{P(u)-P(v)}=
 \frac{|S|(4C^2+1)}
 {|s_us_v|(4c_u^2+1)(4c_v^2+1)}.
\end{equation}
Every term on the right is rational.  Moreover,
\[
 S=\frac{2(u-v)(1+5uv)}
 {(1+5u^2)(1+5v^2)}\neq0
\]
for distinct positive $u,v$, so $P$ is injective on
$\Q\cap(0,1/6)$, and hence its image there is infinite.

\medskip\noindent\emph{General position.}
Up to an invertible real linear transformation, the source curve is
\[
 q(\phi)=(\cos3\phi,\sin3\phi,2\cos2\phi,2\sin2\phi),
 \qquad \phi=4\arctan(\sqrt5u).
\]
Take five distinct source parameters
$0\leq u_0<\cdots<u_4\leq1/6$.  Apply the corresponding rotations in
both frequency blocks so that the first parameter is sent to $0$, write
$\vartheta_j=\phi_j-\phi_0$, and set
$y_j=\tan(\vartheta_j/2)$ for $1\leq j\leq4$.  After subtracting the
first point, define
\[
 v(\vartheta)=
 \bigl(\cos3\vartheta-1,\ \sin3\vartheta,
 2(\cos2\vartheta-1),\ 2\sin2\vartheta\bigr)
\]
and write the affine determinant as
\[
 \mathcal D=\det\bigl(v_\ell(\vartheta_j)\bigr)_{1\leq j,\ell\leq4}.
\]
Half-angle expansion gives the following exact factorization.  Writing
$e_k=e_k(y_1,y_2,y_3,y_4)$ for the elementary symmetric polynomials, set
\[
 \begin{aligned}
 R:={}&5e_1^2-2e_1e_3-3e_2^2-2e_2e_4-50e_2\\
    &\quad-3e_3^2+5e_4^2-70e_4-75.
 \end{aligned}
\]
Then
\begin{equation}\label{eq:explicit-det}
 \mathcal D=-64
 \left(\prod_{j=1}^4\frac{2y_j}{(1+y_j^2)^3}\right)
 \left(\prod_{1\leq i<j\leq4}(y_i-y_j)\right)R.
\end{equation}
Using $e_1^2=\sum_{j=1}^4y_j^2+2e_2$ gives
\[
 \begin{aligned}
 R={}&-75+5\sum_{j=1}^4y_j^2+5e_4^2\\
 &-\bigl(40e_2+2e_1e_3+3e_2^2+2e_2e_4+3e_3^2+70e_4\bigr).
 \end{aligned}
\]
The parameter range gives
\[
 0<y_j\leq\tan\frac{\phi(1/6)}2
 =\frac{12\sqrt5}{31}<1.
\]
Thus $e_k>0$ for $1\leq k\leq4$, and hence
$R<-75+20+5=-50$, while every other factor in
\eqref{eq:explicit-det} is nonzero.  Thus every five points on the
closed source arc, including the pole $N$, are affinely independent.
Lemma~\ref{lem:Q-projection-transfer} now proves that every four
projected points are affinely independent and no five are cospherical.
\end{proof}

\begin{proof}[Proof of Theorem~\ref{thm:primitive-intro}]
By the anchor-point argument in
Remark~\ref{rem:similarity-types}, fixing $n-1$ points of the curve
in Theorem~\ref{thm:P3} and varying the last rational parameter yields
infinitely many pairwise non-similar $n$-point subsets.  For each such
finite set $Y\subset\Q^3$, choose an integer
$L$ clearing the distance denominators and let
\[
 g=\gcd\{L\norm{x-y}:x,y\in Y,\ x\neq y\}.
\]
Fix $y_0\in Y$.  Then
$M=(L/g)(Y-y_0)\subset\Q^3$ contains the origin, has integral distances,
and those distances have gcd one.  The rational-to-lattice embedding
theorem of Marshall and Perlis \cite[Theorem~6]{MarshallPerlis2013}
gives a rotation carrying $M$ into $\Z^3$.
Congruence and scaling preserve affine independence, the property that no
five points are cospherical, and similarity type.  Applying this
construction to one representative from each of the infinitely many
similarity types proves the theorem.\footnote{A concrete primitive
$8_3$-cluster, its exact construction, and its verification are recorded in
the computational supplement and accompanying program included as ancillary
material.}
\end{proof}

For example, choose distinct sufficiently large integers
$M_1,\ldots,M_n$ and take $u_j=1/M_j$ in the permitted interval.  Clearing
denominators and rescaling by the reciprocal of the greatest common divisor
of the pairwise distances then gives an effectively computable, though very crude, upper
bound for the minimum diameter of primitive $n$-point integral point sets in
semi-general position in $\R^3$.  We do not record the formula because no
optimization is attempted.  For the classical minimum-diameter problem for
integral point sets, see \cite{HarborthKemnitzMoller1993}.

%% file: 06-boundaries.tex
\section{Characteristic, rational realization, and limitations}
\label{sec:boundaries}

\subsection{Characteristic and rational realization}

For a nondegenerate $d$-simplex with rational edge lengths, call the class
of its edge Gram determinant in $\Q_{>0}^{\times}/\Q^{\times2}$ its
\emph{characteristic}; for an integral simplex one usually chooses the
positive squarefree integral representative.

\begin{lemma}[Common characteristic and exact rational realization]
\label{lem:rational-realization}
Let $P\subset\R^d$ be a full-dimensional rational distance set, and let
$G$ be the edge Gram matrix of any nondegenerate $d$-simplex in $P$.
The square class of $\det G$ is independent of the chosen simplex and is
the common characteristic of $P$.  Moreover, $P$ admits a congruent
realization in $\Q^d$ if and only if
\[
 G\cong_\Q I_d.
\]
Equivalently, the common characteristic is $1$ and
$\epsilon_p(G)=1$ for every finite prime $p$.
\end{lemma}

\begin{proof}
Fix affinely independent anchors $p_0,\ldots,p_d$, put
$v_i=p_i-p_0$, and let $G=(\ip{v_i}{v_j})$.  Polarization expresses every
entry of $G$ in terms of squared rational distances, so $G$ is rational.
The same calculation shows that every point of $P$ has rational affine
coordinates in the basis $v_1,\ldots,v_d$.  If
$M\in\operatorname{GL}_d(\Q)$ is the edge-coordinate matrix of another
nondegenerate simplex, then
\[
 \det(M^{\mathsf T}GM)=\det(M)^2\det G.
\]
Thus the square class is independent of the simplex, as in the standard
common-characteristic theorem \cite[Theorem~4]{Kurz2006}.

If $G=T^{\mathsf T}T$ with $T\in\operatorname{GL}_d(\Q)$, use the columns
of $T$ as the new edge vectors; the rational affine coordinates just found
place all remaining points in $\Q^d$.  Conversely, rational edge vectors
in a $\Q^d$-realization give $G=C^{\mathsf T}C$ for some
$C\in\operatorname{GL}_d(\Q)$.  The final equivalence follows from
Lemma~\ref{lem:HM}.
\end{proof}

The original even-dimensional sphere curve of
Proposition~\ref{prop:even-geometry} already has characteristic $1$,
even before its quadratic form is identified over $\Q$ with the standard
Euclidean form.  Indeed, choose $2r+1$ affinely independent
rational-parameter points and let $M\in\operatorname{GL}_{2r}(\Q)$ be the
edge matrix before the coordinate weights are applied.  With
\[
 \Lambda=\diag\bigl((\sqrt{\lambda_k},\sqrt{\lambda_k})_{k\in K_r}\bigr),
\]
the weighted edge matrix is $E=\Lambda M$, and hence
\[
 \det(E^{\mathsf T}E)
 =\left(\left(\prod_{k\in K_r}\lambda_k\right)\det M\right)^2
 \in\Q^{\times2}.
\]
The following finite example shows that characteristic $1$ alone does not
imply rational realizability.

\begin{proposition}[Characteristic $1$ without rational realization]
\label{prop:char-counterexample}
There is a full-dimensional nine-point integral point set
$X\subset\R^8$ of characteristic $1$ that admits no congruent realization
in $\Q^8$.
\end{proposition}

\begin{proof}
The Chebyshev identity
\begin{equation}\label{eq:char-counterexample-identity}
 (20U_0+8U_2+5U_4)^2
 =120U_1^2+264U_2^2+80U_3^2+25U_4^2
\end{equation}
gives the positive-definite rational quadratic form with paired diagonal
coefficients
\[
 q=\langle120,120,264,264,80,80,25,25\rangle.
\]
Put
\[
 A=120\cdot264\cdot80\cdot25
   =63\,360\,000=11\cdot2400^2.
\]
Thus $\det q=A^2$ is a square.  Lemma~\ref{lem:paired-invariants},
with $D=1$, gives $\epsilon_v(q)=(A,-1)_v$.  At $v=11$,
\[
 \epsilon_{11}(q)=(11,-1)_{11}
 =\left(\frac{-1}{11}\right)=-1,
\]
whereas $\epsilon_{11}(I_8)=1$.  Hence
$q\not\cong_\Q I_8$.

For $t\in\Q$, set $z(t)=(1+it)/(1-it)$ and
\[
 v(t)=\bigl(
 \Re z(t)^4,\Im z(t)^4,
 \Re z(t)^6,\Im z(t)^6,
 \Re z(t)^8,\Im z(t)^8,
 \Re z(t)^{10},\Im z(t)^{10}\bigr)\in\Q^8.
\]
Define
\[
 \mathcal V:=
 \begin{pmatrix}
  1&v(0)^{\mathsf T}\\
  \vdots&\vdots\\
  1&v(8)^{\mathsf T}
 \end{pmatrix}
 \in M_9(\Q).
\]
Let
\[
 \mathcal D=\diag(\sqrt{120},\sqrt{120},
 \sqrt{264},\sqrt{264},\sqrt{80},\sqrt{80},5,5)
\]
define $P(t)=\mathcal Dv(t)$, and take the nine points
$P_j=P(j)$, $0\leq j\leq8$.  If
$z(t)\overline{z(t')}=c+is$, then
\eqref{eq:char-counterexample-identity} gives
\begin{equation}\label{eq:char-counterexample-distance}
 \norm{P(t)-P(t')}
 =2|s|\,(80c^4-28c^2+17)\in\Q.
\end{equation}
The augmented coordinate matrix $\mathcal V$ is nonsingular.  Since
$1+j^2\not\equiv0\pmod{23}$ for $0\leq j\leq8$, entrywise
reduction defines $\overline{\mathcal V}\in M_9(\mathbb F_{23})$, and an
exact calculation gives
\begin{equation}\label{eq:char-counterexample-det}
 \det\overline{\mathcal V}=10\in\mathbb F_{23}.
\end{equation}
Hence the nine $P_j$ are affinely independent.

The common denominator of their 36 rational distances divides
\[
 L=5^{14}13^{10}17^5 37^5.
\]
Consequently $X=\{LP_0,\ldots,LP_8\}$ is an integral point set.  If
$M=[v(1)-v(0)\mid\cdots\mid v(8)-v(0)]$, then
$M\in\operatorname{GL}_8(\Q)$ and an
edge Gram matrix of $X$ is
\[
 G=L^2M^{\mathsf T}qM.
\]
Its determinant is a rational square, so $X$ has characteristic $1$.  If
$X$ had a congruent realization in $\Q^8$, Lemma
\ref{lem:rational-realization} would give $G\cong_\Q I_8$; since
$LM$ is rational and invertible, this would imply
$q\cong_\Q I_8$, a contradiction.\footnote{The exact data and calculations
verifying the polynomial identity, affine determinant, and rationality of all
$36$ distances are recorded in the computational supplement included as
ancillary material.}
\end{proof}

\begin{remark}
Kurz's theorem \cite[Theorem~4]{Kurz2006} establishes the
common-characteristic property.  Fricke
\cite[Proposition~3.3]{Fricke2001} proves rational-coordinate realizability
for Heron simplices; by definition, the contents of all of their faces of
every dimension are integral.  Kohnert and Kurz
\cite[p.~2107]{KohnertKurz2009}, citing earlier sources, state that
characteristic $1$ yields a rational-coordinate realization.  For an
arbitrary full-dimensional integral point set in higher dimension, the
determinant square class alone does not determine the rational isometry
class of its edge Gram form: its Hasse invariants at all finite places must
also agree with those of the standard Euclidean form.  Thus, at this level
of generality, the statement requires the corresponding local hypotheses.
Proposition~\ref{prop:char-counterexample} gives an explicit example
illustrating the need for this additional local condition, while
Lemma~\ref{lem:rational-realization} gives the exact criterion used
here.  This does not affect the common-characteristic theorem or Fricke's
result.  We do not claim that dimension eight is minimal for this
phenomenon.
\end{remark}

\subsection{Limits of the construction}

\paragraph{General position for even $d$.}
In even dimension $d$, the present harmonic curve lies on a single sphere.
Hence every $d+2$ of its points are cospherical, and this method yields
affine general position but not general position in the sense of
Definition~\ref{def:positions}.  In odd dimension we instead project from
the distinguished pole $N$, which is an endpoint of the source curve arc.
Under this
stereographic projection a sphere in the target hyperplane pulls back to a
hyperplane section of the source sphere, so the short-arc affine
nondegeneracy excludes $d+2$ cospherical image points.  A different
even-dimensional general-position construction is not obtained by the
present method.

\paragraph{\texorpdfstring{$d\equiv3\pmod4$}{d congruent to 3 mod 4}.}
For $d\equiv1\pmod4$, our effective choice of weights and norm parameter
makes the restricted form $Q_D|_{N^{\perp_{Q_D}}}$ have square determinant
and trivial Hasse
invariants at every finite place.  For $d=2r-1\equiv3\pmod4$, put
\[
 \rho=\sum_k\lambda_k,
 \qquad A=\prod_k\lambda_k.
\]
Choosing $D=\rho$ removes the determinant obstruction, but the residual
conditions are
\[
 (A,-\rho)_v(\rho,-1)_v^{\binom r2}=1
 \qquad\text{for every place }v.
\]
They form a simultaneous Hilbert-symbol selection problem.  The positivity
recursion for the Chebyshev weights does not by itself supply a rational
parameter satisfying all these local conditions: Hilbert symbols depend on
square classes and are not continuous under a small positive real
perturbation.  Consequently, verifying these conditions in finitely many
individual dimensions does not yield a uniform theorem for the entire
congruence class.

\paragraph{Primitive lattice normalization.}
For a finite rational-coordinate configuration, clearing coordinate and
distance denominators already gives a lattice realization after scaling.
The additional point in dimension three is primitive normalization: after
rescaling the configuration by the reciprocal of the greatest common
divisor of its distances, the resulting coordinates are a priori only
rational, and Marshall and Perlis
\cite[Theorem~6]{MarshallPerlis2013} restore a congruent realization in
$\Z^3$.
Their paper also treats the planar analogue, but it does not provide a
rational-to-lattice embedding theorem valid in all dimensions.  No uniform
substitute in higher dimensions is proved here.